\documentclass[twoside,11pt,reqno]{amsart}
\usepackage[a4paper,margin=1.3in]{geometry}
\usepackage[dvipsnames]{xcolor}
\usepackage{amsthm,amsmath,amsfonts,graphicx,geometry,lipsum,amsxtra}
\usepackage{hyperref}
\usepackage{mathrsfs}
\usepackage{verbatim}
\usepackage{tikz}
\usepackage{epigraph}
\usepackage{tcolorbox}

\numberwithin{equation}{section}
\usepackage{esint}
\usepackage{mathtools}

\usepackage[dvipsnames]{xcolor}
\usepackage{etoolbox}
\hypersetup{
    colorlinks=true,
    linkcolor=blue,
    citecolor=blue,
    urlcolor=red,
    pdftitle={[BMSS25]Shape-Resonance in Spectral density, Scattering Cross-section, Time delay and Bound on Sojourn time},
    }
\usepackage{comment}
\usepackage{amsmath}
\usepackage{tikz}
\usepackage{epigraph}
\usepackage{lipsum}
\allowdisplaybreaks

\definecolor{titlepagecolor}{cmyk}{1,.60,0,.40}
\patchcmd{\subsection}{\normalfont}{\normalfont\color{blue}}{}{}
\DeclareFixedFont{\titlefont}{T1}{ppl}{b}{it}{0.5in}

\def\th@plain{%
  \thm@notefont{}
  \itshape 
}
\def\th@definition{%
  \thm@notefont{}
  \normalfont 
}
\makeatother

\usepackage{srcltx} 
\usepackage{amscd}
\usepackage{amssymb}
\usepackage{amsmath}
\usepackage{pb-diagram}
\usepackage{graphicx}
\usepackage{hyperref}
\usepackage{array}
\usepackage[all]{xy}
\xyoption{matrix}
\xyoption{arrow}
\usepackage{pdfsync}
\usepackage{physics}
\usepackage{enumitem}
 
\usepackage{thmtools}
\usepackage{amssymb}

\DeclareUnicodeCharacter{2212}{-,+}
 \numberwithin{equation}{section}
\newtheorem{theorem}{Theorem}[section]
\newtheorem{proposition}[theorem]{Proposition}
\newtheorem{corollary}[theorem]{Corollary}
\newtheorem{lemma}[theorem]{Lemma}
\newtheorem{remark}[theorem]{Remark}
\newtheorem{definition}[theorem]{Definition}

\newcommand{\bdfn}{\begin{definition}}
\newcommand{\bthm}{\begin{theorem}}
\newcommand{\blem }{\begin{lemma}}
\newcommand{\bcla }{\begin{cla}aim}
\newcommand{\ecla }{\end{cla}im}
\newcommand{\bpro}{\begin{proposition}}
\newcommand{\bcor }{\begin{corollary}}

\newcommand{\brmrk}{\begin{remark}}
\newcommand{\bassu}{\begin{assumption}}
\newcommand{\eassu}{\end{assumption}}
\newcommand{\edfn}{\end{definition}}
\newcommand{\ethm}{\end{theorem}}
\newcommand{\elem }{\end{lemma}}
\newcommand{\epro}{\end{proposition}}
\newcommand{\ecor}{\end{corollary}}
\newcommand{\ermrk}{\end{remark}}

\newcommand*\diff{\mathop{}\!\mathrm{d}}
\newcommand{\e}{\mathrm{e}}

\newcommand{\RR}{\mathbb R}
\renewcommand{\H}{\mathcal H}

\newcommand{\iu}{{i\mkern1mu}}

\renewcommand{\Im}{\operatorname{Im}}
\DeclareMathOperator{\sgn}{sgn}
\newcommand{\slim}{\mathop{\mathrm{s\text{-}lim}}}
\calclayout
\begin{document}
\title[Shape-Resonance in Spectral density, Scattering Cross-section,$\cdots$]{Shape-Resonance in Spectral density, Scattering Cross-section, Time delay and Bound on Sojourn time}

\author[Bansal]{Hemant Bansal}
\address{Hemant Bansal\\ Department of Mathematical Sciences, Indian Institute of Science Education and Research Mohali\\Sector 81, SAS Nagar, Punjab 140306, India}
\email{ph20030@iisermohali.ac.in}

\author[Maharana]{Alok Maharana}
\address{Alok Maharana\\ Department of Mathematical Sciences, Indian Institute of Science Education and Research Mohali\\Sector 81, SAS Nagar, Punjab 140306, India}
\email{maharana@iisermohali.ac.in}

\author[Sahu]{Lingaraj Sahu}
\address{Lingaraj Sahu\\ Department of Mathematical Sciences, Indian Institute of Science Education and Research Mohali\\Sector 81, SAS Nagar, Punjab 140306, India}
\email{lingaraj@iisermohali.ac.in}

\author[Sinha]{Kalyan B. Sinha}
\address{Kalyan B. Sinha \\
Theoretical Sciences Unit \\
Jawaharlal Nehru Centre for Advanced Scientific Research \\
Jakkur, Bangalore, Karnataka 560064, India}
\email{kbs@jncasr.ac.in}
\vspace*{-1cm}

 \maketitle
\begin{abstract}
The Friedrichs model~\cite{Friedrichs} is revisited to obtain precise results about the asymptotic behaviour (the so-called Breit-Wigner formula~\cite{Breit}) of a resonance near an embedded eigenvalue and the ``spectral concentration" results as a corollary. Some of the abstract results involved can also be used to address similar questions about a rank-one perturbation of the Laplacian. Exact asymptotic properties are also obtained for the sojourn time, the scattering amplitude and time delay.
\end{abstract}
\section{Introduction}
Resonance in quantum theory is a well-studied phenomenon in both time-dependent and time-independent scattering theory. An embedded eigenvalue of a self-adjoint operator may disappear under ``small" perturbations and give rise to resonance near the embedded eigenvalue of the unperturbed operator. This behaviour is often characterized by the local properties of various quantities near the embedded eigenvalue, such as spectral density, time delay, sojourn time and scattering amplitude. More generally, it can also be described in terms of the spectral concentration. Several methods have been developed to study resonances, such as: (a) the formal perturbation series for eigenvalues~\cite{Rejto,Greenlee,KatoBook};~(b)~meromorphic continuation of suitable matrix elements of resolvent from upper to lower half of the complex plane \cite{Titchmarsh,Howland2} or looking at the boundary behaviour of matrix elements of resolvent in upper half-plane~\cite{Howland1}; (c)~theory of dilation analytic  potentials (also known as complex-scaling method)~\cite{Simon1,SimonBook5}; (d) analytic continuation of the scattering matrix~\cite{Newton,KBS4}.
As a simplified approximation, using rank-one perturbation of a self-adjoint operator, the resonance phenomenon has been studied in ~\cite{Friedrichs,Howland1,KBS3,Fernandez1}. Revisiting the Friedrichs model $H_\alpha=M_x+\alpha\langle\cdot,u\rangle u$ on $L^2(\RR),$ we have studied the resonance phenomenon in terms of detailed and finer analysis of the asymptotic behaviour of the spectral density leading to the spectral concentration, in terms of the Cauchy distribution (often called Breit-Wigner formula by physicists~\cite{Breit}). Furthermore, similar detailed analysis can be carried through for the scattering amplitude and time delay by a limiting procedure after suitable translation and scaling of ``energy". We also derive a lower bound for the sojourn time under the perturbed evolutions unlike a possible imprecise limiting result as in~\cite{Fernandez1}. The results from this model in $L^2(\RR)$ are extended to rank-one perturbation of Laplacian in $L^2(\RR^3).$\par
The organization of the paper is as follows. In Section \ref{2}, we present preliminaries results and fix the notations for the subsequent analysis. In Section \ref{3}, we discuss the spectral theory of rank-one perturbations of the multiplication operator $M_x$ in $L^2(\RR)$, establishing necessary and  sufficient conditions for the existence of embedded eigenvalues and we fix the model. Section \ref{4} examines asymptotic behaviour of the spectral density under suitable translation  and scaling. In Section \ref{5}, we derive the estimates for the sojourn time. In Section \ref{6}, we discuss the behaviour of the scattering amplitude as the perturbed scattering system approaches the initial scattering system and obtain asymptotics for time delay. Section \ref{7} deals with the study of the resonance phenomenon for rank-one perturbation of the Laplacian in   $\mathbb R^3.$
\section{Notations and preliminaries}\label{2}
For a self-adjoint operator $T$ (possibly unbounded) on a separable Hilbert space $\H$, let $\sigma(T)$, $\sigma_{ac}(T)$ and $\sigma_{s}(T)$ denote the spectrum, the absolutely continuous spectrum and the singular spectrum of $T$ respectively. Let $\mathcal{H}_{\text{ac}}(T)$ denote the absolutely continuous subspace with respect to $T$. Let $\mathcal B(\H)$, $\mathcal B_1(\H)$ and $\mathcal B_2(\H)$ denote the spaces of bounded, trace class and Hilbert-Schimdt operators on $\H$ respectively. The trace and  Hilbert-Schimdt norms are denoted by $||\cdot||_1$ and $||\cdot||_2$ respectively. A sequence $\{T_n\}\subset\mathcal B(\H)$ is said to converge strongly to an operator $T\in\mathcal B(\H)$, written as $\slim\limits_{n\to\infty}~T_n=T$, if for each $v\in\H$, $T_nv\to Tv$ as $n\to\infty.$\par Let  \( C^\infty(\mathbb{R}) \), \( C^\infty_0(\mathbb{R}) \) and $\mathcal{S}(\mathbb{R})$ denote the spaces of all smooth functions, smooth functions vanishing at $\pm\infty$ and rapidly decreasing smooth functions on $\RR$ respectively. For $1 \leq p\leq \infty$, let $L^p(\mathbb{R})$ denote the standard Lebesgue space.\par  The Fourier transform of a function $f$ is defined as:
\[\widehat{f}(\xi) = \frac{1}{\sqrt{2\pi}} \int_{\mathbb{R}} f(x) e^{-\iu \xi x}\diff x \quad \text{for}\ \xi \in\RR\]
whenever the integral converges, in an appropriate sense.\\
We now present some preliminary results that will be required for computing the resolvent limits of the multiplication operator $M_x$.
\bdfn\label{PVdefn} Let $f \in \mathcal{S}(\mathbb{R})$. The Cauchy principal value of $f$, denoted by $\gamma(f,\cdot)$, is defined as
\begin{equation}\gamma(f,\lambda): = \lim_{\epsilon \to 0^+} \int_{\mathbb{R} \setminus (\lambda - \epsilon, \lambda + \epsilon)} \frac{f(x)}{x - \lambda}\diff x\quad \text{ for }\lambda\in\RR\end{equation} where the above limit exists~$($see Corollary~$4.95$~$($b$)$ in \cite{DorinaMitreaBook}$)$. \edfn
\bpro \label{pro2}Let $f \in \mathcal{S}(\mathbb{R})$. Then the following properties hold: \begin{enumerate}[label=(\alph*)]
    \item The function $\gamma(f,\cdot)$ satisfies\begin{equation}\label{PVFTdefn}\gamma(f,\lambda)=\frac{1}{2\iu}\int_\RR\widehat{f}(\xi)\sgn(\xi)\e^{\iu \xi\lambda}\diff \xi\quad\text{for }\lambda\in\RR\end{equation} where $\sgn(\xi)=1$ for $\xi>0$ and $\sgn(\xi)=-1$ for $\xi<0.$
    \item $\gamma(f,\cdot)\in C^\infty_0(\RR)$ and its $k$-th derivative satisfies \begin{equation}\label{derivativeprincipalvalue}
\frac{\partial^k\gamma(f,\cdot)}{\partial\lambda^k}=\gamma\left(\frac{\diff^kf}{\diff\lambda^k},\cdot\right)
\end{equation}
\item $\lim\limits_{|\lambda|\to\infty}\lambda\,\gamma(f,\lambda)=-\int_\RR f(x)\diff x$
\item For every $\lambda\in\RR$, the conjugation property holds:  \[\overline{\gamma(f,\lambda)}=\gamma(\overline{f},\lambda)\]
\item $\gamma(f,\cdot)\in L^2(\RR)$ and $||\gamma(f,\cdot)||_{L^2(\RR)}=||f||_{L^2(\RR)}$
\end{enumerate}
\epro\begin{proof} For the proof of part~(a), see Corollary~$4.95$~$($c$)$ in \cite{DorinaMitreaBook}. Note that, by part~(a), we have: \begin{align*}\gamma\left(\frac{\diff^kf}{\diff\lambda^k},\lambda\right)&= \frac{1}{2\iu} \int_{\mathbb{R}} \widehat{\frac{\diff^kf}{\diff\lambda^k}}(\xi)\sgn(\xi) \e^{\iu \xi \lambda}\diff \xi\\&=\frac{1}{2\iu} \int_{\mathbb{R}} \widehat{f}(\xi)\sgn(\xi)(\iu\xi)^k \e^{\iu \xi \lambda}\diff \xi\\&=\frac{1}{2\iu} \int_{\mathbb{R}} \widehat{f}(\xi)\sgn(\xi) \frac{\partial^k}{\partial\lambda^k}(\e^{\iu \xi \lambda})\diff \xi\end{align*}
To complete the proof of part~(b), we note that, since $f\in\mathcal S(\RR),$ the derivative and the integral may be interchanged repeatedly by the Lebesgue dominated convergence theorem, which implies that $\gamma(f,\cdot)\in C^\infty(\RR)$. Furthermore, by the Riemann–Lebesgue lemma, we have $\gamma(f,\cdot)\to 0$ as $|\lambda|\to\infty$. Hence, $\gamma(f,\cdot)\in C_0^\infty(\RR).$
By definition of $\gamma$ in \eqref{PVdefn}, we have \[\lambda\, \gamma(f,\lambda)= -\int_\RR f(x)\diff x+\gamma(q,\lambda)\] where $q(x)=xf(x)$ and part~(c) follows as $\gamma(q,\cdot)\in C_0(\RR).$ Part (d) follows from part~(a) and the definition of the Fourier transform. Part~(e) follows from part~(a) and an application of Plancherel's theorem. \end{proof}
\bpro\label{derivative} Let $f\in\mathcal S(\RR)$ such that $f(x_0)=0$. Set \[g(x)=\begin{cases}
\frac{f(x)}{x-x_0} & \text{if } x\neq x_0\\
f'(x_0) & \text{if } x = x_0
\end{cases}\] Then the following properties hold:
\begin{enumerate}[label=(\alph*)]
\item $g\in \mathcal S(\RR)$
\item $\frac{\partial}{\partial\lambda}\gamma(|f|^2,\cdot)_{|_{\lambda=x_0}}=\int_\RR|g(x)|^2\diff x$
    \end{enumerate}
\epro\begin{proof} For $|x|$ large enough, $g(x)$ behaves like $f(x)$, hence $g$ is a function of rapid decrease at $\pm\infty$. Since $f(x_0)=0$ and $f\in C^1(\RR),$ $g\in C^1(\RR).$ By repeating the same argument, we conclude $g\in C^\infty(\RR).$ Hence $g\in S(\mathcal{\RR})$ which proves (a). Now for $x\neq x_0,$ \[(f\overline{g})'(x)=\frac{(|f|^2)'(x)}{x-x_0}-\frac{|f(x)|^2}{(x-x_0)^2}\] Since $f(x_0)=0$, both terms on the right hand side of the above equation are integrable. As $f\overline{g}\in\mathcal S(\RR)$, we obtain \[\int_\RR\frac{(|f|^2)'(x)}{x-x_0}\diff x=\int_\RR\frac{|f(x)|^2}{(x-x_0)^2}\diff x\] This completes the proof of part~(b).
\end{proof}
 \bpro[Plemelj-Privalov Theorem]\label{Plemelj-Privalov Theorem}  Let $f\in\mathcal{S(\RR)}.$ Then \[\lim_{\epsilon\to0^+}\int_\RR\frac{f(x)}{x-(\lambda\pm\iu\epsilon)}\diff x=\gamma(f,\lambda)\pm\iu\pi f(\lambda)\] uniformly for $\lambda$ in compact sets. \epro \noindent For the proof, see equation (4.7.48) in \cite{DorinaMitreaBook}.\\  
 \section{The model}\label{3}
Consider the multiplication operator $M_x$  on the Hilbert space $L^2(\RR)$ with the domain: \[D(M_x)=\left\{f\in L^2(\RR)\Big|\int_{-\infty}^{\infty}|xf(x)|^2\diff x<\infty\right\}\] and \[(M_xf)(x)=xf(x)\quad\text{for }f\in D(M_x)\] Then $M_x$ is a self-adjoint operator with $\H_{ac}(M_x)=L^2(\RR)$. Let $u\in\mathcal S(\RR)$ with ${||u||_{L^2(\RR)}=1}.$ For $\alpha\in\mathbb R,$ consider  the perturbed operator \[H_\alpha:=M_x+\alpha\langle\cdot,u\rangle u\quad\text{ on }L^2(\RR)\] The operator $H_\alpha$ is clearly self-adjoint with the domain $D(H_\alpha) = D(M_x)$.
We denote the resolvent and spectral measure of $H_\alpha$ by $R_\alpha$ and $E_\alpha$ respectively. To analyze the spectrum of $H_\alpha$, we examine the relationship between the matrix elements of the resolvents of $H_\alpha$ and $H_0$ in the following lemma. \blem For any $v,w\in L^2(\RR)$ and $\Im z\neq0$, the following relation holds: \begin{equation}\label{ResolventRelation2}\langle  R_\alpha(z)v,w\rangle =\langle  R_0(z)v,w \rangle-\alpha\dfrac{\langle  R_0(z)v,u \rangle\langle  R_0(z)u,w\rangle}{1+\alpha\langle  R_0(z)u,u \rangle}\end{equation}  \elem \begin{proof} By the resolvent identity, \[R_\alpha(z)-R_0(z)=-\alpha\langle R_0(z)\cdot,u\rangle R_\alpha(z)u\]
and we obtain  \begin{equation}\label{eq1}\langle R_\alpha(z)v,w\rangle=\langle R_0(z)v,w\rangle-\alpha\langle R_0(z)v,u\rangle\langle R_\alpha(z)u,w\rangle\end{equation} which leads to
\begin{equation}\label{ResolventRelation1}\langle R_\alpha(z)u,w\rangle=\dfrac{\langle  R_0(z)u,w\rangle}{1+\alpha\langle  R_0(z)u,u \rangle}\end{equation}
Substituting \eqref{ResolventRelation1} into \eqref{eq1}, we obtain \eqref{ResolventRelation2}.\end{proof}In order to study the spectral properties of the self-adjoint operator \( H_\alpha \), in the next lemma, we analyze the boundary values of the terms appearing in \eqref{ResolventRelation2} as \( \Im z \to 0 \) from upper and lower half of the complex plane.
 \blem\label{exception} Let $v,w\in\mathcal S(\RR)$. Set \[\langle R_0(\lambda\pm\iu0)v,w\rangle:=\lim_{\epsilon\to0^+}\langle  R_0(\lambda\pm\iu\epsilon)v,w\rangle, \text{ if it exists}; \text{ and}\] \[\Gamma_\alpha^\pm=\{\lambda\in\RR~|~1+\alpha \langle R_0(\lambda\pm\iu0)u,u\rangle=0\}\]
Then \begin{enumerate}[label=(\alph*)]\item $\langle R_0(\lambda\pm\iu0)v,w\rangle=\gamma(v\overline{w},\lambda)\pm\iu\pi v(\lambda)\overline{w(\lambda)}$, uniformly for $\lambda$ in compact sets.\item $\Gamma_\alpha^+=\Gamma_\alpha^-$, denoted by $\Gamma_\alpha$, is a closed set of Lebesgue measure $0$ and $\sigma_s(H_\alpha)\subset\Gamma_\alpha$.
\item For $\lambda\in\RR\setminus\Gamma_\alpha,$ the following holds \[ \lim_{\epsilon\to0^+}\langle R_\alpha(\lambda\pm\iu\epsilon)v,w\rangle=\langle  R_0(\lambda\pm\iu0)v,w \rangle-\alpha\dfrac{\langle  R_0(\lambda\pm\iu0)v,u \rangle\langle  R_0(\lambda\pm\iu0)u,w\rangle}{1+\alpha\langle  R_0(\lambda\pm\iu0)u,u \rangle}\] where the above convergence is uniform for $\lambda$ in compact subsets of $\RR\setminus\Gamma_\alpha$.
\end{enumerate} \elem \begin{proof}
Part (a) follows by an application of Proposition \ref{Plemelj-Privalov Theorem}. $\Gamma_\alpha^+=\Gamma_\alpha^-$ follows from part~(a). For the proof of rest of part~(b), see Lemma 9.5 in \cite{KBSBook}. Part (c) follows from the definition of $\Gamma_\alpha$, part~(a) and equation \eqref{ResolventRelation2}.
\end{proof}
In this model, since the perturbation is compact (in fact rank-one), the essential spectrum of $H_\alpha$ is entire real line. Therefore any eigenvalue (if any) of $H_\alpha$ is naturally embedded in the continuous spectrum and we study such a case.
Define \begin{equation}\label{F}F(\alpha,\lambda):=1+\alpha \langle R_0(\lambda+\iu0)u,u\rangle=F_1(\alpha,\lambda)+\iu F_2(\alpha,\lambda)\end{equation} where \begin{equation}\label{F1F2defn}
    F_1(\alpha,\lambda)=1+\alpha\gamma(|u|^2,\lambda)\quad \text{and}\quad  F_2(\alpha,\lambda)=\alpha\pi|u(\lambda)|^2
\end{equation}
\begin{theorem}\label{EmbeddedDisappear}Let $u\in\mathcal S(\RR)$ be fixed. Then
  \begin{enumerate}[label=(\alph*)]
    \item The operator $H_{\alpha_0}$ has an eigenvalue $\lambda_0$ if and only if $u({\lambda_0})=0$ and $1+\alpha_0\gamma(|{u}|^2,\lambda_0)=0$. In this case, $\lambda_0$ is a simple eigenvalue with the eigenvector \begin{equation}\label{phi}\phi(x)=\begin{cases}
\frac{u(x)}{x-\lambda_0} & \text{if } x\neq\lambda_0 \\
u'(\lambda_0) & \text{if } x = \lambda_0
\end{cases}\end{equation}  \item Suppose $u$ vanishes only at $x=\lambda_0$ and $1+\alpha_0\gamma(|{u}|^2,\lambda_0)=0$. Then $\mathcal{H}_{\text{ac}}(H_{\alpha_0}) = \{\phi\}^\perp$ and $\H_{ac}(H_\alpha)=L^2(\RR)$ for $\alpha\neq\alpha_0$.
    \end{enumerate}
\end{theorem}
\begin{proof}
 By Lemma \ref{exception}~(b), it follows that if $\lambda_0$ is an eigenvalue of $H_{\alpha_0}$, then $\lambda_0\in\Gamma_{\alpha_0}$ which implies $u(\lambda_0)=1+\alpha_0\gamma(|u|^2,\lambda_0)=0.$ Conversely, since $u \in \mathcal{S}(\mathbb{R})$ and $u(\lambda_0) = 0$, by Proposition \ref{derivative}(a), the function $\phi$ defined in \eqref{phi} is in $\mathcal S(\RR)$ and hence $\phi\in D(H_{\alpha_0}).$ Given the condition $1+\alpha_0\gamma(|u|^2,\lambda_0)=0,$ it follows that \[(H_{\alpha_0}-\lambda_0)\phi=\left(1+\alpha_0\gamma(|u|^2,\lambda_0)\right)u=0\] Also, if $\psi$ is any other eigenvector of $H_{\alpha_0}$ corresponding to $\lambda_0$, then \[(M_x-\lambda_0)\psi+\alpha_0\langle\psi,u\rangle u=0\]which implies \[ \psi(x)=-\alpha_0\langle\psi,u\rangle \frac{u(x)}{x-\lambda_0}\quad\text{ for a.e. }x\in\RR\] Hence $\lambda_0$ is a simple eigenvalue of $H_{\alpha_0}$. \par To prove part~(b), first observe that since $F_2(\alpha_0,\lambda)=\alpha_0\pi|u(\lambda)|^2\neq0$ for $\lambda\neq\lambda_0$, thus $\Gamma_{\alpha_0}=\{\lambda_0\}.$ Hence $\lim\limits_{\epsilon\to0^+}\langle R_{\alpha_0}(\lambda+\iu\epsilon)v,v\rangle$ exists for any $v\in \mathcal{S}(\RR)$ and $\lambda\neq\lambda_0$. Since $\mathcal S(\RR)$ is dense in $L^2(\RR)$, by Corollary 3.28 in \cite{Teschl}, it follows that $\H_{ac}(H_{\alpha_0})=\{\phi\}^\perp$. Now for $\alpha\neq\alpha_0$, \begin{align*}F_2(\alpha,\lambda)\neq0\quad\text{for }\lambda\neq\lambda_0\quad\text{ and }\quad F_1(\alpha,\lambda_0)=1+\alpha\gamma({|u|^2},\lambda_0)=1-\frac{\alpha}{\alpha_0}\neq0\end{align*} Hence $\Gamma_\alpha=\emptyset$ for all $\alpha\neq\alpha_0$ and we get $\H_{ac}(H_\alpha)=L^2(\RR).$ This completes the proof of part~(b).
\end{proof}\noindent We now define the model that forms the basis for the subsequent analysis.
\begin{definition}[Model in $\RR$]\label{assumptions}Let $\alpha_0,\lambda_0\in\RR$ be fixed. For $\alpha\in\RR,$ consider \[H_\alpha=M_x+\alpha\langle\cdot,u\rangle u\quad\text{on }L^2(\RR)\] where $u$ satisfies: $u\in\mathcal S(\RR)$ with $||u||_{L^2(\RR)}=1,$ $u$ vanishes only at $x=\lambda_0$ and $1+\alpha_0\gamma(|u|^2,\lambda_0)=0.$
 \end{definition}
\noindent In the model defined above, by Theorem \ref{EmbeddedDisappear}, $\lambda_0$ is a simple eigenvalue of $H_{\alpha_0}$ embedded in the absolutely continuous spectrum which disappears after perturbation. All the results that follow will be with respect to this model.
\section{Spectral density and its asymptotic properties}\label{4}
In our model, since \( H_\alpha \) is purely absolutely continuous for \( \alpha \neq \alpha_0 \), the spectral density associated with the measure $\langle E_\alpha(\diff\lambda)v,w\rangle$ for $v,w\in L^2(\RR)$ is given by \begin{equation}\label{spectraldensity}\begin{split}
    \rho^{v,w}(\alpha,\lambda)&=\frac{\diff}{\diff\lambda}\langle E_\alpha(\diff\lambda)v,w\rangle\\&=\lim_{\epsilon\to0^+}\frac{1}{2\pi\iu}\left(\langle R_\alpha(\lambda+\iu\epsilon)v,w\rangle-\langle R_\alpha(\lambda-\iu\epsilon)v,w\rangle\right)\end{split}\end{equation} for almost every $\lambda\in\RR.$
    Here we compute the density $\rho^{v,w}(\alpha,\cdot)$ explicitly for our model and study its asymptotic behaviour in terms of the Cauchy distribution near the embedded eigenvalue $\lambda_0$ as $\alpha\to\alpha_0$ and then we obtain the spectral concentration and time decay.\blem\label{SpecDe}For the operators $H_\alpha$, $\alpha\in\RR,$ we have\begin{enumerate}[label=(\alph*)]\item For $v,w\in \mathcal S(\RR)$ and $\alpha\neq\alpha_0$, the spectral density $\rho^{v,w}(\alpha,\cdot)$ is given by:\begin{equation}\label{DensityForvw} \begin{split} \rho^{v,w}(\alpha,\lambda)=&v(\lambda)\,\overline{w(\lambda)}-\alpha\frac{F_1(\alpha,\lambda)}{|F(\alpha,\lambda)|^2}\left(\gamma({v\overline{u}},\lambda)\,u(\lambda)\,\overline{w(\lambda)}+\gamma({u\overline{w}},\lambda) \,v(\lambda)\,\overline{u(\lambda)}\right)\\& +\alpha\frac{F_2(\alpha,\lambda)}{|F(\alpha,\lambda)|^2}\left(-\pi v(\lambda)\,\overline{w(\lambda)}\,|u(\lambda)|^2+\frac{1}{\pi}\,\gamma({v\overline{u}},\lambda)\,\gamma({u\overline{w}},\lambda)\right)\end{split}\end{equation} for $\lambda\in\RR.$
    \item  There exists a $C^1$ function  $\lambda(\alpha)$ defined on an open interval $J$ around  $\alpha_0$ such that $F_1(\alpha,\lambda(\alpha))=0$ for $\alpha\in J$ and $|\lambda(\alpha)-\lambda_0|=O(|\alpha-\alpha_0|)$ as $\alpha\to\alpha_0$. 
\end{enumerate} \elem\begin{proof}Part (a) follows from \eqref{spectraldensity} and Lemma \ref{exception}. Note that  the function $F_1(\alpha,\lambda)=1+\alpha\gamma({|u|^2},\lambda)$ is a smooth mapping from $\RR^2$ to\ $\RR$ and $F(\alpha_0,\lambda_0)=0.$
Also using Lemma \ref{derivative} (b), we get \begin{equation}\label{gammaderivative}\frac{\partial F_1}{\partial\lambda}(\alpha_0,\lambda_0)=\alpha_0\gamma((|u|^2)',\lambda_0)=\alpha_0||\phi||^2\neq0\end{equation}
Thus by the implicit function theorem, there exists an open interval $J$ around $\alpha_0$ and a $C^1$ function $\lambda:J\to\RR$ such that $\lambda(\alpha_0)=\lambda_0$ and $F_1(\alpha,\lambda(\alpha))=0$. Also \[\lambda'(\alpha_0)=-\dfrac{\frac{\partial F_1}{\partial\alpha}(\alpha_0,\lambda_0)}{\frac{\partial F_1}{\partial\lambda}(\alpha_0,\lambda_0)}=\frac{1}{\alpha_0^2||\phi||^2}\]
 Since $\lambda'(\alpha_0)\neq 0$, $|\lambda(\alpha)-\lambda_0|=O(|\alpha-\alpha_0|)$ as $\alpha\to\alpha_0.$ This completes the proof of part~(b). \end{proof}  Now define the function $\kappa:J\to\RR$ by
\begin{equation}\label{kappa(alpha)}\displaystyle\kappa(\alpha):=\frac{\pi|u(\lambda(\alpha))|^2}{||\phi||^2}\end{equation} Then  $\kappa(\alpha_0)=0$ and for $\alpha\in J\setminus\{\alpha_0\}$, \(\kappa(\alpha) > 0\). As $u(\lambda_0)=0$ and \[|u(\lambda(\alpha))-u(\lambda_0)|^2=\left|\int_{\lambda_0}^{\lambda(\alpha)}u'(x)\diff x\right|^2\leq ||u'||_\infty^2|\lambda(\alpha)-\lambda_0|^2\] there exists $C'>0$ such that $
\kappa(\alpha)\leq C'|\alpha-\alpha_0|^2 $ for $\alpha$ near $\alpha_0$. If $u$ vanishes to order $n$ at $\lambda_0$, then repeating the above argument, there exists a constant $C>0$ such that 

\begin{equation}\label{kappaestimateR^1}\kappa(\alpha)\leq C|\alpha-\alpha_0|^{2n}\quad\text{ for }\alpha\text{ near }\alpha_0\end{equation}  For simplicity of notation, we set \begin{equation}
    \lambda_h(\alpha)=\lambda(\alpha)+h\kappa(\alpha)\quad\text{for}\ \alpha\in J,~h\in\RR
\end{equation}
\blem \label{lemma1}For any fixed $h\in\RR$, we have \begin{enumerate}[label=(\alph*)]
 \item\begin{equation}\label{F1F2}\lim_{\alpha\to\alpha_0}\frac{F_1(\alpha,\lambda_h(\alpha))}{\kappa(\alpha)}=h\alpha_0||\phi||^2\quad \text{and}\quad  \lim_{\alpha\to\alpha_0}\frac{F_2(\alpha,\lambda_h(\alpha))}{\kappa(\alpha)}=\alpha_0||\phi||^2\end{equation} \item \begin{equation}\label{F1limit}\lim_{\alpha\to\alpha_0}\frac{\kappa(\alpha)F_1(\alpha,\lambda_h(\alpha))}{|F(\alpha,\lambda_h(\alpha))|^2}=\frac{1}{\alpha_0||\phi||^2}\frac{h}{h^2+1}\end{equation}\item\begin{equation}\label{F2limit}\lim_{\alpha\to\alpha_0}\frac{\kappa(\alpha)F_2(\alpha,\lambda_h(\alpha))}{|F(\alpha,\lambda_h(\alpha))|^2}=\frac{1}{\alpha_0||\phi||^2}\frac{1}{h^2+1}\end{equation}\end{enumerate}\elem
 \begin{proof}
      To prove part~(a), we use the following fact: for $q\in C^1(\RR)$,
\begin{equation}\label{limit1}\lim_{\alpha\to\alpha_0}\frac{q(\lambda_h(\alpha))-q(\lambda(\alpha))}{\kappa(\alpha)}=hq'(\lambda_0)\end{equation}
       Since $F_1(\alpha,\lambda(\alpha))=0$, we have \[\frac{F_1(\alpha,\lambda_h(\alpha))}{\kappa(\alpha)}=\frac{F_1(\alpha,\lambda_h(\alpha))-F_1(\alpha,\lambda(\alpha))}{\kappa(\alpha)}=\alpha\frac{\gamma({|u|^2},\lambda_h(\alpha))-\gamma({|u|^2},\lambda(\alpha))}{\kappa(\alpha)}\] Now applying \eqref{limit1} to the function $\gamma(|u|^2,\cdot)$ and using \eqref{gammaderivative}, we obtain
  \[\lim_{\alpha\to\alpha_0}\frac{F_1(\alpha,\lambda_h(\alpha))}{\kappa(\alpha)}=h\alpha_0\gamma((|u|^2)',\lambda_0)=h\alpha_0||\phi||^2\] Noting that by \eqref{limit1}, \[\lim_{\alpha\to\alpha_0}\frac{|u(\lambda_h(\alpha))|^2-|u(\lambda(\alpha))|^2}{\kappa(\alpha)}=h(|u|^2)'(\lambda_0)=0\] which implies
       \[\lim_{\alpha\to\alpha_0}\frac{|u(\lambda_h(\alpha))|^2}{\kappa(\alpha)}=\frac{||\phi||^2}{\pi}\] Thus, by the definition of $F_2$ in \eqref{F1F2defn}, we obtain
\[\lim_{\alpha\to\alpha_0}\frac{F_2(\alpha,\lambda_h(\alpha))}{\kappa(\alpha)}=\alpha_0||\phi||^2\]This complete the proof of (a). For part~(b), we write \[\frac{\kappa(\alpha)F_1(\alpha,\lambda_h(\alpha))}{|F(\alpha,\lambda_h(\alpha))|^2}=\dfrac{\dfrac{F_1(\alpha,\lambda_h(\alpha))}{\kappa(\alpha)}}{\left(\dfrac{F_1(\alpha,\lambda_h(\alpha))}{\kappa(\alpha)}\right)^2+\left(\dfrac{F_2(\alpha,\lambda_h(\alpha))}{\kappa(\alpha)}\right)^2}\] and using part~(a), part~(b) follows. The proof of part~(c) follows similarly.\end{proof}The next theorem gives a precise result on the behaviour of density as $\alpha\to\alpha_0$ in the context of the simple model. The appearance of the Cauchy distribution should be compared with Theorem 3 in \cite{Howland1}.
 \begin{theorem}\label{thm4}  Let $v,w\in\mathcal S(\RR).$ Then for any fixed $h\in \mathbb R,$ we have \begin{equation}\label{vw asym}\lim_{\alpha\to \alpha_0}\kappa(\alpha)\rho^{v,w}(\alpha,\lambda_h(\alpha))=\frac{1}{\pi}\frac{1}{h^2+1}\langle P_{\lambda_0}v,w\rangle\end{equation}where $P_{\lambda_0}$ denotes the eigen-projection corresponding to eigenvalue $\lambda_0$ of $H_{\alpha_0}.$ In particular, \begin{equation}\label{evasym}\lim_{\alpha\to \alpha_0}\kappa(\alpha)\rho^{\phi}(\alpha,\lambda_h(\alpha))=\frac{||\phi||^2}{\pi}\frac{1}{h^2+1}\end{equation} \end{theorem}
    \begin{proof}
   For fixed $h\in\RR,$ substituting $\lambda=\lambda_h(\alpha)$ in \eqref{DensityForvw} and multiplying it by $\kappa(\alpha)$, we get \begin{equation}\label{A1234}\begin{aligned} \kappa(\alpha)\rho^{v,w}(\alpha,\lambda_h(\alpha))=&\kappa(\alpha)\,v\left(\lambda_h(\alpha)\right)\,\overline{w(\lambda_h(\alpha))}\\[5pt]&\quad-\alpha\frac{\kappa(\alpha)\,F_1(\alpha,\lambda_h(\alpha))}{|F(\alpha,\lambda_h(\alpha))|^2}\left(\gamma({v\overline{u}},\lambda_h(\alpha))\, u(\lambda_h(\alpha))\,\overline{w(\lambda_h(\alpha))}\right)\\[5pt]&\quad+\alpha\frac{\kappa(\alpha)\,F_1(\alpha,\lambda_h(\alpha))}{|F(\alpha,\lambda_h(\alpha))|^2}\left(\gamma({u\overline{w}},\lambda_h(\alpha)) \,v(\lambda_h(\alpha))\,\overline{u(\lambda_h(\alpha))}\right)\\[5pt]&\quad-\alpha\pi\frac{\kappa(\alpha)\,F_2(\alpha,\lambda_h(\alpha))}{|F(\alpha,\lambda_h(\alpha))|^2}\left( v(\lambda_h(\alpha))\,\overline{w(\lambda_h(\alpha))}\,|u(\lambda_h(\alpha))|^2\right)\\[5pt]&\quad+\frac{\alpha}{\pi}\frac{\kappa(\alpha)\,F_2(\alpha,\lambda_h(\alpha))}{|F(\alpha,\lambda_h(\alpha))|^2}\left(\gamma({v\overline{u}},\lambda_h(\alpha))\,\gamma({u\overline{w}},\lambda_h(\alpha))\right)\end{aligned}\end{equation}
    To obtain \eqref{vw asym}, we show that, the last term give the desired limit and other terms goes to $0$ as $\alpha\to\alpha_0.$
    Since $\lambda_h(\alpha)\to\lambda_0,\ \kappa(\alpha)\to0$ as $\alpha\to\alpha_0,$ we observe that
\[\lim_{\alpha\to\alpha_0}\kappa(\alpha)v(\lambda_h(\alpha))\overline{w(\lambda_h(\alpha))}=0\cdot v(\lambda_0)\cdot\overline{w(\lambda_0)}=0\] Since $\lim\limits_{\alpha\to\alpha_0}u(\lambda_h(\alpha))=0$, by using \eqref{F1limit} and \eqref{F2limit}, the second, third and fourth terms in \eqref{A1234} go to $0$ as $\alpha\to\alpha_0$. Also,\[\lim_{\alpha\to\alpha_0}\gamma({v\overline{u}},\lambda_h(\alpha))=\gamma({v\overline{u}},\lambda_0)=\langle v,\phi\rangle\hspace{2mm}\text{and}\hspace{2mm}\lim_{\alpha\to\alpha_0}\gamma({u\overline{w}},\lambda_h(\alpha))=\gamma({u\overline{w}},\lambda_0)=\langle \phi ,w\rangle\]
Hence again using \eqref{F2limit}, we get \begin{align*} &\lim_{\alpha\to\alpha_0}\frac{\alpha}{\pi}\frac{\kappa(\alpha)F_2(\alpha,\lambda_h(\alpha))}{|F(\alpha,\lambda_h(\alpha))|^2}\left(\gamma({v\overline{u}},\lambda_h(\alpha))\gamma({u\overline{w}},\lambda_h(\alpha))\right)\\
&=\frac{\alpha_0}{\pi}\left(\frac{1}{\alpha_0||\phi||^2}\frac{1}{h^2+1}\right)\left(\langle v,\phi\rangle\langle \phi,w\rangle\right)\\&=\frac{1}{\pi}\frac{1}{h^2+1}\langle P_{\lambda_0}v,w\rangle\end{align*}
\end{proof}
 Next we consider the phenomenon of spectral concentration of $H_\alpha$ near $\lambda_0$ as $\alpha\to\alpha_0.$ Let us recall the definition of spectral concentration from \cite{KatoBook}.
 \bdfn[Spectral Concentration] Let $\H$ be a Hilbert space and $\{T_\alpha\}_{\alpha\in\RR}$ be a family of self-adjoint operators in $\H$ with the associated spectral measures $\left\{E_\alpha\right\}_{\alpha\in\RR}$. Let $\lambda_0$ be an eigenvalue of $T_{\alpha_0}.$ We say that the spectrum of $T_\alpha$ is concentrated near $\lambda_0$ as $\alpha\to\alpha_0$ if there exist $\{J_\alpha\}$, a family of Borel subsets of $\mathbb R$ such that the Lebesgue measure of $J_\alpha$, $|J_\alpha|\to 0$ and $E_\alpha(J_\alpha)\xrightarrow{} E_{\alpha_0}(\{\lambda_0\})$ strongly as $\alpha\to\alpha_0$. If in addition, $\lim\limits_{\alpha\to\alpha_0}\frac{| J_\alpha|}{|\alpha-\alpha_0|^p}=0$, then we say that the spectrum of $\{T_\alpha\}$ is concentrated to order $p$ near $\lambda_0.$ \edfn The following proposition is useful in deducing the spectral concentration which can be proved by simply following the proof of Theorem 1.15, Chapter VIII in \cite{KatoBook}.
 \bpro\label{pro} Let $T_\alpha$ converges to $T_{\alpha_0}$ in the strong resolvent sense as $\alpha\to\alpha_0$. Let $b(\alpha)$ be a real-valued function continuous at $\alpha_0$ and $P=E_{\alpha_0}(\{b(\alpha_0)\}).$ Then
 \[\slim_{\alpha\to\alpha_0}E_\alpha(-\infty,b(\alpha)](1-P)= E(-\infty,b(\alpha_0))\]
\epro
 Theorem \ref{thm4} and Proposition \ref{pro} together implies the spectral concentration of $H_{\alpha}$ near the embedded eigenvalue $\lambda_0$ of $H_{\alpha_0}$~(cf. Theorem 4 in \cite{Howland1}).
\bcor\label{SpecCor} Suppose $u$ vanishes to order $n$ at $x=\lambda_0.$ Then for any $p\in[0,2n),$ the spectrum of $H_\alpha$ is concentrated near $\lambda_0$ to order $p$ as $\alpha\to\alpha_0$.\ecor
\begin{proof} Let $p\in[0,2n).$ Now choose $r$ such that $p<r<2n$ and let\\ $J_\alpha=\left(\lambda(\alpha)-(\kappa(\alpha))^{\frac{r}{2n}},\lambda(\alpha)+(\kappa(\alpha))^{\frac{r}{2n}}\right].$ 
By the resolvent identity, for $\Im z\neq0,$ \[\lim_{\alpha\to\alpha_0}||R_\alpha(z)-R_{\alpha_0}(z)||_1=0\] in particular $R_\alpha(z)\to R_{\alpha_0}(z)$ in the strong resolvent sense. By Proposition \ref{pro}, it follows that   \begin{equation}\label{eq2}\slim_{\alpha\to\alpha_0}E_\alpha(J_\alpha)(I-P_{\lambda_0})=0\end{equation}
On the other hand,
    \begin{align*}\langle E_\alpha(J_\alpha)\phi,\phi\rangle&=\int_{J_\alpha}\rho^\phi(\alpha,\lambda)\diff\lambda=\int_{-\frac{|J_\alpha|}{2\kappa(\alpha)}}^{\frac{|J_\alpha|}{2\kappa(\alpha)}}\kappa(\alpha)\,\rho^\phi(\alpha,\lambda_h(\alpha))\diff h\end{align*} Set
\[f_\alpha(h):=\chi_{\left(-\frac{|J_\alpha|}{2\kappa(\alpha)},\frac{|J_\alpha|}{2\kappa(\alpha)}\right)}(h)\,\kappa(\alpha)\,\rho^\phi(\alpha,\lambda_h(\alpha)),\ f(h):=\frac{||\phi||^2}{\pi}\frac{1}{h^2+1}\] where $\chi_A$ denote the indicator function of set $A$.
  Since
    \[\lim_{\alpha\to\alpha_0}\frac{|J_\alpha|}{2\kappa(\alpha)}=\lim_{\alpha\to\alpha_0}\frac{1}{\kappa(\alpha)^{1-\frac{r}{2k}}}=\infty\]  then by \eqref{evasym} it follows that $f_\alpha\to f$ pointwise as $\alpha\to\alpha_0.$ Furthermore, \[\int_\RR f_\alpha(h)\diff h\leq ||\phi||^2=\int_\RR f(h)\diff h\]
    Then by Lemma \ref{Halmos}, it follows that \[\lim_{\alpha\to\alpha_0}\int_\RR f_\alpha(h)\diff h=\int_\RR f(h)\diff h\]
    which implies \begin{equation}\label{ev}\lim_{\alpha\to\alpha_0}\langle E_\alpha(J_\alpha)\phi,\phi\rangle=||\phi||^2=\langle P_{\lambda_0}\phi,\phi\rangle\end{equation}
   Now for any $v,w\in L^2(\RR)$, we have \begin{align*}\langle E_\alpha(J_\alpha)P_{\lambda_0}v,w\rangle&=\langle E_\alpha(J_\alpha)P_{\lambda_0}v,P_{\lambda_0}w\rangle+\langle E_\alpha(J_\alpha)P_{\lambda_0}v,(I-P_{\lambda_0})w\rangle\\&=\frac{\langle P_{\lambda_0}v,w\rangle}{||\phi||^2}\langle E_\alpha(J_\alpha)\phi,\phi\rangle+\langle P_{\lambda_0}v,E_\alpha(J_\alpha)(I-P_{\lambda_0})w\rangle\end{align*}
   By \eqref{eq2} and \eqref{ev}, it follows that $\langle E_\alpha(J_\alpha)P_{\lambda_0}v,w\rangle\to\langle P_{\lambda_0}v,w\rangle$ as $\alpha\to\alpha_0.$ Since limiting operator is a projection, weak convergence implies strong convergence and we get $E_\alpha(J_\alpha)P_{\lambda_0}\to P_{\lambda_0}$ strongly as $\alpha\to\alpha_0.$ Hence by \eqref{eq2}, we conclude that  \[\slim_{\alpha\to\alpha_0}E_\alpha(J_\alpha)=P_{\lambda_0}\]
 Also by \eqref{kappaestimateR^1}, it follows that \[\frac{|J_\alpha|}{|\alpha-\alpha_0|^p}=\frac{2\kappa(\alpha)^{\frac{r}{2n}}}{|\alpha-\alpha_0|^p}\leq \frac{2C^{\frac{r}{2n}}|\alpha-\alpha_0|^r}{|\alpha-\alpha_0|^p}\]which implies $\lim\limits_{\alpha\to\alpha_0}\dfrac{|J_\alpha|}{|\alpha-\alpha_0|^p}=0.$ Hence, we conclude that the spectrum of $H_\alpha$ is concentrated near $\lambda_0$ to order $p$ as $\alpha\to\alpha_0.$
\end{proof}

Now, we want to study a finer structure of the evolution generated by $H_\alpha$ by shifting the mean and scaling the time near $\alpha=\alpha_0.$ The following result is an immediate consequence of \eqref{evasym}. See \cite{Davies} for various implications related to the behaviour of spectral density  and spectral concentration.
\begin{corollary} \label{exp} For the eigenvector $\phi$ of $H_{\alpha_0}$, \[\lim_{\alpha\to\alpha_0}\langle\e^{-\iu t\frac{H_\alpha-\lambda(\alpha)}{\kappa(\alpha)}}\phi,\phi\rangle=||\phi||^2\e^{-|t|}\] uniformly in $t.$\end{corollary}
 \begin{proof}
 For $h\in\RR,$ define \[g_\alpha(h):=\kappa(\alpha)\rho^\phi(\alpha,\lambda_h(\alpha)),\  g(h)=\frac{||\phi||^2}{\pi}\frac{1}{h^2+1}\] Then by \eqref{vw asym}, $g_\alpha\to g$ pointwise as $\alpha\to\alpha_0.$
    Observe that, \[\int_\RR g_\alpha(h)\diff h=\int_\RR g(h)\diff h=||\phi||^2\]
    So, by Lemma \ref{Halmos}, we conclude that $g_\alpha\to g$ in $L^1$ norm as $\alpha\to\alpha_0.$ This also implies that  $\widehat{g_\alpha}\to\widehat{g}$ uniformly on $\RR.$ Note that $\sqrt{2\pi}\hat{g}(t)=||\phi||^2\e^{-|t|}$  and
     \begin{align*}\langle \e^{-\iu t\frac{H_\alpha-\lambda(\alpha)}{\kappa(\alpha)}}\phi,\phi \rangle&=\int_\RR\e^{-\iu t\frac{\lambda-\lambda(\alpha)}{\kappa(\alpha)}}\rho^\phi(\alpha,\lambda)\diff\lambda\\&=\int_\RR\e^{-\iu th}\kappa(\alpha)\rho^{\phi}(\alpha,\lambda_h(\alpha))\diff h\\&=\sqrt{2\pi}\hat{g_\alpha}(t)\end{align*}
 Thus $\lim\limits_{\alpha\to\alpha_0}\langle\e^{-\iu t\frac{H_\alpha-\lambda(\alpha)}{\kappa(\alpha)}}\phi,\phi\rangle=||\phi||^2\e^{-|t|}$ uniformly in $t$.\end{proof}

\section{Sojourn time and its properties} \label{5}

The sojourn time for a state $v\in L^2(\RR)$ with respect to the time evolution generated by the self-adjoint operator $H_\alpha$ is given by:
\[\tau_\alpha(v):=\int_{-\infty}^{\infty}|\langle\e^{-\iu tH_\alpha}v,v\rangle|^2\diff t\] In the context of our model, in the next theorem, we analyze  $\int_{-\infty}^{\infty}|\langle\e^{-\iu tH_\alpha}v,w\rangle|^2\diff t$ for a pair of vectors $v,w$ in a suitable dense subset of $L^2(\RR)$, for $\alpha\neq\alpha_0$.
\begin{theorem}\label{S1} Let $v,w\in\mathcal S(\RR).$ Then for any $\alpha\neq\alpha_0$, \[\tau_{\alpha}(v,w):=\int_{-\infty}^{\infty}|\langle\e^{-\iu tH_\alpha}v,w\rangle|^2\diff t<\infty\]  \end{theorem}
 \begin{proof}
For $\alpha\neq\alpha_0,$ by an application of the Plancherel's theorem,
\[\tau_\alpha(v,w)=\int_\RR\left|\int_\RR\e^{-\iu t\lambda}\rho^{v,w}(\alpha,\lambda)\diff\lambda\right|^2\diff t=2\pi\int_\RR|{\rho^{v,w}}(\alpha,\lambda)|^2\diff\lambda\]
Since $u\in S(\RR)$ and $\gamma(|u|^2,\cdot)\in C_0(\RR)$,  $\lim\limits_{\lambda\to\pm\infty} F(\alpha,\lambda)=1.$ As $F(\alpha,\lambda)$ never vanishes for any $\lambda$, there exists a constant $c_\alpha>0$ such that $|F(\alpha,\lambda)|\geq c_\alpha>0$ for all $\lambda\in\RR.$ Since all the functions $u,v,w,\gamma(u\overline{v},\cdot),\gamma({v\overline{w}},\cdot),\gamma({u\overline{w}})$ are bounded, it follows from the expression of $\rho^{v,w}(\alpha,\cdot)$ in \eqref{DensityForvw} that there exists a constant $\Tilde{c_\alpha}>0$ satisfying $|\rho^{v,w}(\alpha,\lambda)|\leq \Tilde{c_\alpha}$ for all $\lambda\in\RR.$ Since $\rho^{v, w}(\alpha,\cdot) \in L^1(\RR)$, it follows that  $\rho^{v, w}(\alpha,\cdot) \in L^2(\RR)$, which completes the proof.\end{proof}
In particular, for the eigenvector $ \phi$, we have $\tau_\alpha(\phi)<\infty$ for $\alpha\neq\alpha_0$. Note that $\tau_{\alpha_0}(\phi)=\infty$ and in the next theorem we derive a lower bound for $\tau_\alpha(\phi)$ for $\alpha$ near $\alpha_0.$
 \begin{theorem}\label{Sojourn}
     There exists $\delta>0$ such that \[\tau_\alpha(\phi)>\frac{||\phi||^4}{4\kappa(\alpha)}\ \ \text{for}\ |\alpha-\alpha_0|<\delta\]
 \end{theorem}
 \begin{proof}
For $\alpha\in J\setminus\{\alpha_0\}$ and $T\in(0,\infty),$ define \begin{equation}\label{truncated}
 \tau_\alpha(\phi,T):=\int_{-T}^T|\langle\e^{-\iu tH_\alpha}\phi,\phi\rangle|^2\diff t.
\end{equation} Then $\lim\limits_{T\to\infty}\tau_\alpha(\phi, T)=\tau_\alpha(\phi).$ Replacing $T$ by $T/\kappa(\alpha)$ in \eqref{truncated}, we get
\[\kappa(\alpha)\tau_\alpha\left(\phi,T/\kappa(\alpha)\right)=\int_{-T}^T\left|\left\langle\e^{-\iu t\frac{H_\alpha-\lambda(\alpha)}{\kappa(\alpha)}}\phi,\phi\right\rangle\right|^2\diff t\] Since for any $\alpha$ near $\alpha_0,\, t\in\RR,$ \[\left|\langle\e^{-\iu t\frac{H_\alpha-\lambda(\alpha)}{\kappa(\alpha)}}\phi,\phi\rangle\right|\leq ||\phi||^2\]  by Corollary \ref{exp}, \[\lim_{\alpha\to\alpha_0}\left|\langle\e^{-\iu t\frac{H_\alpha-\lambda(\alpha)}{\kappa(\alpha)}}\phi,\phi\rangle\right|^2=||\phi||^4\e^{-2|t|}\]Thus for any finite $T$, by the Lebesgue dominated convergence theorem, we get\[\lim_{\alpha\to\alpha_0}\kappa(\alpha)\tau_\alpha\left(\phi,T/\kappa(\alpha)\right)=||\phi||^4\int_{-T}^{T}\e^{-2|t|}\diff t=||\phi||^4(1-\e^{-2T})\]Choose $T_0 > 0$ such that $1 - \e^{-2T_0} > \frac{1}{2}$ and for this $T_0$, we have a $\delta > 0$ such that for $|\alpha - \alpha_0|< \delta$, \[\kappa(\alpha) \tau_\alpha\left(\phi, \frac{T_0}{\kappa(\alpha)}\right) > \|\phi\|^4 (1 - \e^{-2T_0}) - \frac{\|\phi\|^4}{4} > \frac{\|\phi\|^4}{4}\]
Since $\tau_\alpha(\phi) \geq \tau_\alpha\left(\phi, \frac{T_0}{\kappa(\alpha)}\right)$, it follows that
$\tau_\alpha(\phi) > \frac{\|\phi\|^4}{4\kappa(\alpha)}.$\end{proof}
 \begin{remark}
     There are mathematical difficulties in obtaining an ``exact limiting form" for the sojourn time $\tau_\alpha(\phi)$ as done in \cite{Fernandez1}, however the lower bound in Theorem \ref{Sojourn} gives an estimate on the rate of divergence of $\tau_\alpha(\phi)$ as $\alpha\to\alpha_0.$
 \end{remark}
\section{Scattering theory and Krein's spectral shift function}\label{6}
\noindent For the pair of operators $H_0$ and $H_\alpha$ with $\alpha\neq\alpha_0$ in our model, note that $H_\alpha-H_0\in\mathcal B_1(L^2(\RR))$. Hence by Theorem 6.2.1 in \cite{Yafaev}, the wave operators \[\Omega^{(\alpha)}_{\pm}:= \slim_{t \to \pm \infty} \e^{\iu tH_\alpha } \e^{-\iu t H_0 }\] exist and are complete. Consider the scattering operator $S^{(\alpha)}=\Omega_+^{(\alpha)*} \Omega_-^{(\alpha)}$ corresponding to the pair $(H_0,H_\alpha).$ Since $S^{(\alpha)}$ is a unitary operator which commutes with $H_0,$ it is a multiplication operator by the function $S^{(\alpha)}(x)$ of modulus 1 for $x\in\RR.$ As $H_0$ is already diagonalized in $L^2(\RR)$, by Theorem 6.7.3 in \cite{Yafaev}, we have \[
    S^{(\alpha)}(x)=1-2\pi\iu\alpha\frac{|u(x)|^2}{F(\alpha,x)}\] Note that right-hand side is well defined as $F(\alpha,x)$ never vanishes for $\alpha\neq\alpha_0.$ Then by \eqref{F1F2defn},\begin{equation}\label{Sexpression}S^{(\alpha)}(x)=1-2\iu \frac{F_2(\alpha,x)}{F(\alpha,x)}=\frac{\overline{F(\alpha,x)}}{F(\alpha,x)}=\e^{-2\pi\iu\xi_\alpha(x)}
\end{equation}where \[\xi_\alpha(x)=\frac{1}{\pi}\arg F(\alpha,x)\] is the  Krein's spectral shift function for the pair of operators $(H_0,H_\alpha)$ and  '$\arg$' denotes the principal branch of the argument function (c.f. \cite{Yafaev} for more details on spectral shift function). We now study the behaviour of $R^{(\alpha)}(x)=S^{(\alpha)}(x)-1$, and that of $|R^{(\alpha)}(x)|^2$ which we interpret as total scattering cross-section at ``energy'' $x$ by analogy with the space-time picture of scattering theory in higher dimensions~\cite{KBSBook}. 
\begin{theorem}\label{thm100}
 For any fixed $h\in\RR,$\begin{equation}\label{11} \lim_{\alpha\to\alpha_0}R^{(\alpha)}(\lambda_h(\alpha))=-\frac{2\iu}{h+\iu}\text{ and }\lim_{\alpha\to\alpha_0}|R^{(\alpha)}(\lambda_h(\alpha))|^2=\frac{4}{h^2+1}\end{equation}
\end{theorem}\begin{proof}Note that
\[R^{(\alpha)}(\lambda_h(\alpha)))=-2\iu\frac{F_2(\alpha,\lambda_h(\alpha))}{F(\alpha,\lambda_h(\alpha))}\] 
By applying \eqref{F1F2} the theorem follows.
\end{proof}
In the following theorem, we obtain the behaviour of $\xi'_\alpha$ near $\lambda_0$ as $\alpha$ goes to $\alpha_0.$
\begin{theorem}\label{ssf thm} For each fixed $h\in\RR,$ \[\lim_{\alpha\to\alpha_0}\kappa(\alpha)\xi_\alpha'(\lambda_h(\alpha))=-\frac{1}{\pi(h^2+1)}\]   \end{theorem}\begin{proof} By~\eqref{Sexpression}, we get
\begin{equation}\label{eq13}\xi_\alpha'(\lambda)=\frac{1}{\pi}\dfrac{F_1(\alpha,\lambda)\frac{\partial F_2}{\partial\lambda}(\alpha,\lambda)-F_2(\alpha,\lambda)\frac{\partial F_1}{\partial\lambda}(\alpha,\lambda)}{|F(\alpha,\lambda)|^2} \end{equation}
Substituting $\lambda=\lambda_h(\alpha)$ in \eqref{eq13} and multiplying it by $\kappa(\alpha),$ we get\[\kappa(\alpha)\xi_\alpha'(\lambda_h(\alpha))=\frac{1}{\pi}\dfrac{\kappa(\alpha)F_1(\alpha,\lambda_h(\alpha))\frac{\partial F_2}{\partial\lambda}(\alpha,\lambda_h(\alpha))}{|F(\alpha,\lambda_h(\alpha))|^2}-\frac{1}{\pi}\frac{\kappa(\alpha)F_2(\alpha,\lambda_h(\alpha))\frac{\partial F_1}{\partial\lambda}(\alpha,\lambda_h(\alpha))}{|F(\alpha,\lambda_h(\alpha))|^2}\]
Now using \eqref{F1limit}, \eqref{F2limit} together with $\dfrac{\partial F_1}{\partial\lambda}(\alpha_0,\lambda_0)=\alpha_0||\phi||^2,~\dfrac{\partial F_2}{\partial\lambda}(\alpha_0,\lambda_0)=0$, the proof follows.
\end{proof}
In analogy with physical scattering theory where the unperturbed operator is Laplacian, which will be discussed in the next section~(see \eqref{TD1}), the time delay in scattering at ``energy" $\lambda$ is given as a multiple of the derivative of spectral shift function. However by analogy with usual space-time formulation, there is no scope for making such a space-time definition in this model (since the ``space" variable looks like the ``energy" variable) and we take $-2\pi\xi_\alpha'(\lambda)$ to be the definition of time delay  at ``energy" $\lambda$.
\section{An application to perturbation of Laplacian in $\RR^3$}\label{7}
Consider the self-adjoint extension $H_0$ of the Laplacian operator $-\Delta=-\sum_{i=1}^3\frac{\partial^2}{\partial x_i^2}$ on the Hilbert space $L^2(\RR^3)$, with  the domain\[D(H_0)=\left\{v\in L^2(\RR^3):\int_{\RR^3}\left||\xi|^2\hat v(\xi)\right|^2<\infty\right\}\] where $\hat v$ is the Fourier transform of $v$ in  $L^2$-sense.
Then $H_0$ is a positive self-adjoint operator. Let $u\in L^2(\RR^3).$ For $\alpha\in\RR,$ define  \[H_\alpha =H_0+\alpha\langle\cdot,u\rangle u\quad\text{ on } L^2(\RR^3)\]
Then  $H_\alpha$ is a self-adjoint operator with $D(H_\alpha)=D(H_0).$ Since the perturbation is rank-one, the essential spectrum of $H_\alpha$ is $[0,\infty)$. We denote the resolvent and spectral measure of $H_\alpha$ by $R_\alpha$ and $E_\alpha$ respectively. In order to apply some of the results from the previous model, it is convenient to place the problem in a unitarily isomorphic Hilbert space $ L^2([0,\infty), L^2(S^2))$ via the isomorphism $U$ defined by\begin{equation}\label{iso}(Uv)_\lambda(\omega)=2^{-\frac{1}{2}}\lambda^{\frac{1}{4}}\widehat{v}(\sqrt{\lambda}\omega)\equiv v_\lambda(\omega)\quad \text{ for almost every }\lambda\in[0,\infty),\omega\in S^2\end{equation} where $S^2$ is the unit sphere in $\RR^3$ with the surface area measure and  the set $[0,\infty)$  with  the Lebesgue measure. We denote the inner product and the norm on $L^2(S^2)$ by  $\langle\langle\cdot,\cdot\rangle\rangle$ and $|||\cdot|||$ respectively. Then \[\langle v,w\rangle=\int_0^\infty \langle\langle v_\lambda,w_\lambda\rangle\rangle\diff\lambda\] 
With a slight abuse of notation, for $v,w\in L^2(\RR^3)$, define 
\[\langle\langle v,w\rangle\rangle(\lambda):=\langle \langle v_\lambda,w_\lambda\rangle\rangle\quad \text{ for almost every } \lambda\in(0,\infty)\] and extend it by setting $0$ for $\lambda\le 0$.
Consider the space $\mathcal{E}((0,\infty), L^2(S^2))$ of all strongly smooth functions $v \in L^2([0,\infty), L^2(S^2))$ such that $|||v_\lambda|||$ and $|||v_{(1/\lambda)}|||$ decreases rapidly as $\lambda\to \infty$. Then $\mathcal{E}((0,\infty), L^2(S^2))$ is dense in $L^2([0,\infty),L^2(S^2)).$ We denote the $n$-th strong derivative $\frac{\diff^n}{\diff\lambda^n}{u_\lambda}$ at $\lambda=s$ by $u^{(n)}_s.$
For $v,w\in \mathcal{E}((0,\infty), L^2(S^2))$, note that the function $\langle\langle v,w\rangle\rangle \in\mathcal S(\RR)$ and by Proposition \ref{Plemelj-Privalov Theorem}, we have  \begin{equation}\lim_{\epsilon\to 0^+}\langle R_0(\lambda\pm\iu\epsilon)v,w\rangle=\lim_{\epsilon\to0^+}\int_\RR\frac{\langle\langle v,w\rangle\rangle(x)}{x-(\lambda+\iu\epsilon)}\diff x=\gamma(\langle\langle v,w\rangle\rangle,\lambda)\pm\iu\pi \langle\langle v_\lambda,w_\lambda\rangle\rangle\end{equation}
 uniformly for $\lambda$ in compact subsets of $\RR$, where $\gamma(f,\cdot)$ is the Cauchy principal value of $f$ defined in \eqref{PVdefn}. Suppose $u\in\mathcal{E}((0,\infty), L^2(S^2))$, then for $\alpha\in\RR,\lambda>0$, define \begin{align*}
     G(\alpha,\lambda):= \lim_{\epsilon\to 0^+}(1+\alpha\langle R_0(\lambda+\iu\epsilon)u,u\rangle)=G_1(\alpha,\lambda)+\iu G_2(\alpha,\lambda)
\end{align*} where
\begin{equation}\label{G} G_1(\alpha,\lambda):=1+\alpha\gamma(\langle\langle u,u\rangle\rangle,\lambda),\ G_2(\alpha,\lambda):=\alpha\pi\langle\langle u,u\rangle\rangle(\lambda)=\alpha\pi|||u_\lambda|||^2\end{equation} We have the following analogue of Theorem \ref{EmbeddedDisappear}.  \begin{theorem} Let $u\in\mathcal E((0,\infty),L^2(S^2))$ be fixed. Then, we have
    \begin{enumerate}[label=(\alph*)]
   \item $\lambda_0>0$ is an eigenvalue of $H_{\alpha_0}$ if and only if $u_{\lambda_0}=0$ as a vector in $L^2(S^2)$ and $1+\alpha_0\gamma(\langle\langle u,u\rangle\rangle,\lambda_0)=0$. In this case, $\lambda_0$ is a simple eigenvalue with the eigenvector $\phi$ where \begin{equation}  \phi_\lambda=\begin{cases}
        \frac{ u_\lambda}{\lambda-\lambda_0} &\text{ if }\lambda\neq\lambda_0\\u'_{\lambda_0}&\text{ if }\lambda=\lambda_0
    \end{cases} \quad\text{ as vectors in }L^2(S^2)\end{equation}Further, $\phi\in\mathcal E((0,\infty),L^2(S^2)).$\item Suppose $\lambda_0,~\alpha_0>0$ and $u_\lambda=0$ in $L^2(S^2)$ only for $\lambda=\lambda_0$ in $(0,\infty)$ and $1+\alpha_0\gamma(\langle\langle u,u\rangle\rangle,\lambda_0)=0$. Then $\mathcal{H}_{\text{ac}}(H_{\alpha_0}) = \{\phi\}^\perp$ and $\H_{ac}(H_\alpha)=L^2(\RR^3)$ for $\alpha\neq\alpha_0$ and $\alpha>0$.
\end{enumerate}
\end{theorem}
\begin{proof}  First observe that, for $\alpha>0,$ $H_\alpha$ can not have any non-negative eigenvalue as $H_\alpha$ is strictly positive. Applying the arguments similar to what was employed in the  proof of Theorem \ref{EmbeddedDisappear}, parts~(a) and (b) follow.
\end{proof} 
\begin{definition}[Model in $\RR^3$]\label{model2} Let $\alpha_0$ and $\lambda_0$ be two fixed positive real numbers. For $\alpha>0$, let \begin{equation}\label{model22}
H_\alpha=H_0+\alpha\langle\cdot,u\rangle u\quad\text{ on }L^2(\RR^3)
\end{equation}
 where $u\in\mathcal E((0,\infty),L^2(S^2))$ with $||u||_{L^2(\RR^3)}=1$ satisfying $u_\lambda=0 $ in $L^2(S^2)$ only for $\lambda=\lambda_0$ in $(0,\infty)$ and, $1+\alpha_0\gamma(\langle\langle u,u\rangle\rangle,\lambda_0)=0$.
\end{definition}  All subsequent results in this section are with respect to this model.
\subsection{Density asymptotics and other spectral properties}
For the  model in the  above definition, we have the following analogues of the results in Section \ref{4} and \ref{5}.
\begin{theorem}For the operators $H_\alpha$, defined in \eqref{model22}, we have the following:\begin{enumerate}[label=(\alph*)]
    \item For $\alpha>0$ with $\alpha\neq\alpha_0$ and  $v,w\in \mathcal E((0,\infty),L^2(S^2))$, the spectral density $\rho^{v,w}(\alpha,\cdot)$ associated with the measure $\langle E_\alpha(\diff\lambda)v,w\rangle$ is given by: \begin{equation}\label{DensityForvw1}\begin{split}\rho^{v,w}(\alpha,\lambda)=&\langle\langle v_\lambda,w_\lambda\rangle\rangle-\alpha\frac{G_1(\alpha,\lambda)}{|G(\alpha,\lambda)|^2}\bigl[\gamma(\langle\langle v,u\rangle\rangle,\lambda)\,\langle\langle u_\lambda,w_\lambda\rangle\rangle+\gamma(\langle\langle u,w\rangle\rangle,\lambda) \,\langle\langle v_\lambda, u_\lambda\rangle\rangle\bigr]\\[5pt]&\quad +\alpha\frac{G_2(\alpha,\lambda)}{|G(\alpha,\lambda)|^2}\left[-\pi \langle\langle v_\lambda,w_\lambda\rangle\rangle\,||| u_\lambda|||^2+\frac{1}{\pi}\gamma(\langle\langle v, u\rangle\rangle,\lambda)\,\gamma(\langle\langle u,w\rangle\rangle,\lambda)\right]\end{split}\end{equation}

    \item  There exists a $C^1$ function  $\lambda(\alpha)$ defined on an open interval $J$ containing  $\alpha_0$ such that $\lambda(\alpha_0)=\lambda_0$ and $G_1(\alpha,\lambda(\alpha))=0$ for $\alpha\in J$ and $|\lambda(\alpha)-\lambda_0|= O(|\alpha-\alpha_0|)$ as $\alpha$ goes to $\alpha_0$. Define the function $\kappa$ on $J$ by
\begin{equation}\label{kappa(alpha)1}\displaystyle\kappa(\alpha):=\frac{\pi|||u_{\lambda(\alpha)}|||^2}{||\phi||^2}\end{equation} Then, $\kappa(\alpha_0)=0$ and $\kappa(\alpha)>0$ for $\alpha\in J\setminus\{\alpha_0\}.$ If $u_\lambda$ vanishes to order $n$ at $\lambda_0$, then there exists $C>0$ such that $\kappa(\alpha)\leq C|\alpha-\alpha_0|^{2n}$ for $\alpha$ near $\alpha_0$. \item Writing $\lambda_h(\alpha)=\lambda(\alpha)+h\kappa(\alpha)$ as earlier,  for any fixed $h\in \mathbb R,$ we have \begin{equation}\label{vw asym1}\lim_{\alpha\to \alpha_0}\kappa(\alpha)\rho^{v,w}(\alpha,\lambda_h(\alpha))=\frac{1}{\pi}\frac{1}{h^2+1}\langle P_{\lambda_0}v,w\rangle\end{equation}where $P_{\lambda_0}$ denotes the eigenprojection corresponding to eigenvalue $\lambda_0$ of $H_{\alpha_0}.$  \item If $u_\lambda$ vanishes to order $n$ at $\lambda_0$, then for any $p<2n,$ the spectrum of $H_\alpha$ is concentrated to order $p$ at $\lambda_0$ as $\alpha$ goes to $\alpha_0.$ \item  For the eigenvector $\phi$ of $H_{\alpha_0}$, \[\lim_{\alpha\to\alpha_0}\langle\e^{-\iu t\frac{H_\alpha-\lambda(\alpha)}{\kappa(\alpha)}}\phi,\phi\rangle=||\phi||^2\e^{-|t|}\] uniformly in $t.$\item For $\alpha>0$ with $\alpha\neq\alpha_0$, we have $\tau_\alpha(v,w):=\int_{-\infty}^{\infty}|\langle\e^{-\iu tH_\alpha}v,w\rangle|^2\diff t<\infty$. Here $\tau_\alpha(\phi):=\tau_\alpha(\phi,\phi)$ is the sojourn time at the vector $\phi$. Moreover, there exists a $\delta>0$ such that\[\tau_\alpha(\phi)>\frac{||\phi||^4}{4\kappa(\alpha)}\ \ \text{for}\ |\alpha-\alpha_0|<\delta\]
\end{enumerate} \end{theorem}\begin{proof} Part~(a) is the analogue of \eqref{DensityForvw} in the previous model and the proof is along the same lines. For part~(b), observe  by Proposition~\ref{derivative}~(b) that  \[\frac{\partial G_1}{\partial\lambda}(\alpha_0,\lambda_0)=\alpha_0\int_\RR\frac{|||u_\lambda|||^2}{(\lambda-\lambda_0)^2}\diff \lambda=\alpha_0||\phi||^2\neq0\] Thus by the implicit function theorem, we have an open interval $J$ around $\alpha_0$ such that there exists a $C^1$ function $\lambda:J\to\RR$ satisfying the desired properties.\\ Since $|||u_\lambda|||$ vanishes only at $\lambda_0$   we have  $\kappa(\alpha_0)=0  $ and $\kappa(\alpha)>0$ for $\alpha\in J\setminus\{\alpha_0\}.$ Also,  we have \[|||u_{\lambda(\alpha)}|||= |||u_{\lambda(\alpha)}-u_{\lambda_0}|||=\left|\left|\left|\int_{\lambda_0}^{\lambda(\alpha)}u'_\lambda\diff\lambda\right|\right|\right|\leq  
\left|\int_{\lambda_0}^{\lambda(\alpha)} |||u'_\lambda ||| \diff\lambda \right|\]
This implies
$|||u_{\lambda(\alpha)}|||^2\le 
|\lambda(\alpha)-\lambda_0|^2~||u'||^2_{L^\infty ((0,\infty),L^2(S^2))}$ which further implies using the properties of $\lambda(\alpha)$ that for some  constant $C'>~0$, $\kappa(\alpha)\leq{C'|\alpha-\alpha_0|^2}$ for $\alpha$ near $\alpha_0.$ When $u_\lambda$ vanishes to order $n$ at $\lambda_0$ we get the desired estimate by  repeating the above argument. This completes the proof of part~(b). By argument as in Lemma \ref{lemma1} (a), for any fixed $h\in\RR$ we have \begin{equation}\label{F1F2 1}\lim_{\alpha\to\alpha_0}\frac{G_1(\alpha,\lambda_h(\alpha))}{\kappa(\alpha)}=h\alpha_0||\phi||^2\text{ and }
        \lim_{\alpha\to\alpha_0}\frac{G_2(\alpha,\lambda_h(\alpha))}{\kappa(\alpha)}=\alpha_0||\phi||^2\end{equation} and (c) follows. The parts (d) and  (e)  follow by similar arguments as in Corollary~\ref{SpecCor} and Corollary~\ref{exp} respectively.

        For any fixed $\alpha>0$ with $\alpha \neq \alpha_0$, the finiteness of $\tau_\alpha(v,w)$ is a consequence of the $L^2$-integrability of $\rho^{v,w}(\alpha,\cdot)$, as in Theorem~\ref{S1}. Proof of part (f) is along the line of argument in Theorem~\ref{Sojourn}.\end{proof}
\subsection{Behaviour of scattering cross-section and time delay}

 Since $H_\alpha-H_0\in\mathcal B_1(L^2(\RR))$, similar to the previous model, the wave operators $\Omega^{(\alpha)}_{\pm}$ for the pair $(H_0,H_\alpha)$ exist and are complete (cf.~section \ref{6} for more comments).  The  scattering operator $S^{(\alpha)}=\Omega_+^{(\alpha)*} \Omega_-^{(\alpha)}$ commutes with $H_0$ and is decomposable with respect to the representation \eqref{iso} of $H_0$, i.e. \[US^{(\alpha)}U^{-1}=\{S_\lambda^{(\alpha)}\}\]For $\lambda\in(0,\infty),$ $S_\lambda^{(\alpha)}$ is a unitary operator on $L^2(S^2)$ and it is called the scattering matrix at ``energy" $\lambda.$  By Theorem 6.7.3 in \cite{Yafaev} (see also Proposition 8.22 in \cite{KBSBook}), we have
 \[ S_\lambda^{(\alpha)}=I-\frac{2\pi\iu\alpha}{G(\alpha,\lambda)}\langle\langle \cdot,u_\lambda\rangle\rangle u_\lambda\] By \eqref{G}, for $\lambda\neq\lambda_0$
\begin{equation}
S_\lambda^{(\alpha)}=I-2\iu\frac{G_2(\alpha,\lambda)}{G(\alpha,\lambda)}\langle\langle\cdot,\frac{u_\lambda}{|||u_\lambda|||}\rangle\rangle\frac{u_\lambda}{|||u_\lambda|||}\end{equation}    
Note that, the determinant \begin{equation}\label{det S}\det  S_\lambda^{(\alpha)}=1-2\iu\frac{G_2(\alpha,\lambda)}{G(\alpha,\lambda)}=\frac{\overline{{G(\alpha,\lambda)}}}{G(\alpha,\lambda)}= \e^{-2\pi\iu\xi_\alpha(\lambda)}\end{equation}
 where $\xi_\alpha(\lambda)=\frac{1}{\pi}\arg G(\alpha,\lambda)$ is the Krein's spectral shift function.
  For $\alpha>0,\alpha\neq\alpha_0,$ define $R^{(\alpha)}_\lambda=S^{(\alpha)}_\lambda-I$, i.e.
  \begin{equation}\label{R expression}R^{(\alpha)}_\lambda=-\frac{2\pi\iu\alpha}{G(\alpha,\lambda)}\langle\langle \cdot,u_\lambda\rangle\rangle u_\lambda\end{equation} 
 In the following theorem, we study the asymptotic behaviour of $R^{(\alpha)}_\lambda$   near $\lambda_0 $  as $\alpha$ goes to  $\alpha_0$.
\begin{theorem}\label{Rthm}Consider the model in Definition \ref{model2} along with additional hypothesis that $u'_{\lambda_0}\neq 0.$ Then for any fixed $h\in\RR,$ \[\lim_{\alpha\to\alpha_0}R^{(\alpha)}_{\lambda_h(\alpha)}=-\frac{2\iu}{h+\iu} \langle\langle\cdot,e({\lambda_0})\rangle\rangle e({\lambda_0})\quad\text{in Hilbert-Schmidt norm}\] where $e(\lambda_0)=u'_{\lambda_0}/|||u'_{\lambda_0}|||.$\end{theorem}
\begin{proof}  Substituting $\lambda=\lambda_h(\alpha) $ in \eqref{R expression}, we get
%
\begin{equation}\label{I_1I_2}
R^{(\alpha)}_{\lambda_h(\alpha)} =-2\iu\cdot I_1(\alpha,h)\cdot I_2(\alpha,h)
\end{equation}
where \[I_1(\alpha,h)= \langle  \langle \cdot,\frac{{u}_{\lambda_h(\alpha)}}{|||{u}_{\lambda_h(\alpha)}|||} \rangle \rangle  \frac{{u}_{\lambda_h(\alpha)}}{|||{u}_{\lambda_h(\alpha)}|||}\quad\text{and}\quad I_2(\alpha,h)=\frac{\alpha\pi|||u_{\lambda_h(\alpha)}|||^2}{G(\alpha,\lambda_h(\alpha))}\]
Noting   by definition  \eqref{G} of $G_2 $ that \begin{align*}I_2(\alpha,h)&=\frac{G_2(\alpha,\lambda_h(\alpha))}{G(\alpha,\lambda_h(\alpha))}=\frac{\kappa(\alpha)}{G(\alpha,\lambda_h(\alpha)) }\frac{G_2(\alpha,\lambda_h(\alpha))}{\kappa(\alpha)}\end{align*}  \eqref{F1F2 1} gives \begin{equation}\label{I2}\lim_{\alpha\to\alpha_0}I_2(\alpha,h)=\frac{1}{h+\iu}\end{equation}
We can rewrite \[I_1(\alpha,h)= \langle  \langle \cdot,\frac{\tilde{u}_{\lambda_h(\alpha)}}{|||{u}_{\lambda_h(\alpha)}|||} \rangle \rangle  \frac{\tilde{u}_{\lambda_h(\alpha)}}{|||{u}_{\lambda_h(\alpha)}|||}\] where $\tilde{u}_{\lambda_h(\alpha)}=\sgn(\lambda_h(\alpha)-\lambda_0)u_{\lambda_h(\alpha)}.$
We have
\begin{equation}\label{estimate}\begin{split}
     |||u_{\lambda_h(\alpha)}-u'_{\lambda_0}(\lambda_h(\alpha)-\lambda_0)|||&=\left|\left|\left|\int^{\lambda_h(\alpha)}_{\lambda_0}(u_\lambda'-u'_{\lambda_0})\diff\lambda\right|\right|\right|\\&=\left|\left|\left|\int^{\lambda_h(\alpha)}_{\lambda_0}\int^{\lambda}_{\lambda_0}u''_s\diff s\diff\lambda\right|\right|\right|\leq c|\lambda_h(\alpha)-\lambda_0|^2\end{split}
 \end{equation} where $c=\|u''\|_{L^{\infty}((0,\infty), L^2(S^2))}.$ By triangle inequality,\begin{equation}\label{triangleinequality}|||u'_{\lambda_0}|||\cdot|\lambda_h(\alpha)-\lambda_0|-
  |||u_{\lambda_h(\alpha)}|||
 \leq c|\lambda_h(\alpha)-\lambda_0|^2
\end{equation}
 For $\alpha\in J\setminus\{\alpha_0\}$, consider\begin{align*}&\frac{\tilde{u}_{\lambda_h(\alpha)}}{|||{u}_{\lambda_h(\alpha)}|||}-e(\lambda_0)=\frac{\tilde{u}_{\lambda_h(\alpha)}}{|||{u}_{\lambda_h(\alpha)}|||}-\frac{u'_{\lambda_0}[\sgn(\lambda_h(\alpha)-\lambda_0)]}{|||u'_{\lambda_0}|||}\frac{\lambda_h(\alpha)-\lambda_0}{|\lambda_h(\alpha)-\lambda_0|}\\&\quad\quad=\frac{\tilde{u}_{\lambda_h(\alpha)}}{|||{u}_{\lambda_h(\alpha)}|||}-\frac{\tilde{u}_{\lambda_h(\alpha)}}{|||u'_{\lambda_0}|||\cdot|\lambda_h(\alpha)-\lambda_0|}+\frac{\tilde{u}_{\lambda_h(\alpha)}-u'_{\lambda_0}[\sgn(\lambda_h(\alpha)-\lambda_0)](\lambda_h(\alpha)-\lambda_0)}{|||u'_{\lambda_0}|||\cdot|\lambda_h(\alpha)-\lambda_0|}\end{align*} By triangle inequality we get
\begin{align*}\left|\left|\left|\frac{\tilde{u}_{\lambda_h(\alpha)}}{|||{u}_{\lambda_h(\alpha)}|||}-e(\lambda_0)\right|\right|\right|&\leq \dfrac{|||u_{\lambda_h(\alpha)}-u'_{\lambda_0}(\lambda_h(\alpha)-\lambda_0)|||}{|||u'_{\lambda_0}|||\cdot|\lambda_h(\alpha)-\lambda_0|}\\&\hspace{4mm}+|||u_{\lambda_h(\alpha)}|||\left|\frac{1}{|||u_{\lambda_h(\alpha)}|||}-\frac{1}{|||u'_{\lambda_0}|||\cdot|\lambda_h(\alpha)-\lambda_0|}\right|\end{align*}  
 As $\alpha\to\alpha_0$, first term goes to 0 by \eqref{estimate} and second term goes to 0 by \eqref{triangleinequality}. Hence, \[\slim_{\alpha\to\alpha_0}\frac{\tilde{u}_{\lambda_h(\alpha)}}{|||{u}_{\lambda_h(\alpha)}|||}= e(\lambda_0)\] and $I_1(\alpha,h)$ converges to $\langle\langle\cdot, e({\lambda_0})\rangle\rangle e({\lambda_0})$ in Hilbert-Schmidt norm as $\alpha\to\alpha_0$ and by \eqref{I2} the result follows.
\end{proof}
\begin{remark}
When  $u_\lambda$ is strongly analytic and vanishing up to order $n$ at $\lambda_0$, then for any fixed $h\in\RR,$ \[ \lim_{\alpha\to\alpha_0}R^{(\alpha)}_{\lambda_h(\alpha)}=-\frac{2\iu}{h+\iu}\langle\langle\cdot, \frac{u^{(n)}_{\lambda_0}}{|||u^{(n)}_{\lambda_0}|||}\rangle\rangle\frac{u^{(n)}_{\lambda_0}}{|||u^{(n)}_{\lambda_0}|||}\quad\text{in Hilbert-Schmidt norm}.\]
\end{remark}
Now we analyze the scattering cross-section and the scattering amplitude for the scattering system $(H_0,H_\alpha)$.
The total scattering cross-section $\overline{\sigma}_\alpha(\lambda)$ at ``energy" $\lambda$ for the scattering system $(H_0,H_\alpha)$ satisfies  (see equation (7.69) in \cite{KBSBook})\[\overline{\sigma}_\alpha(\lambda):=\pi\lambda^{-1}||R^{(\alpha)}_\lambda||^2_{2}\] \noindent Note that, the kernel of $R^{(\alpha)}_\lambda$ is given by \[R^{(\alpha)}_\lambda(\omega_0,\omega)=-2\pi\iu\alpha\frac{u_\lambda(\omega_0)\overline{u_\lambda(\omega)}}{G(\alpha,\lambda)}\quad\text{for a.e. }(\omega_0,\omega)\in S^2\times S^2\]
  We now define the scattering amplitude $f(\lambda;\omega_0\to\omega)$ by (see equation (7.48) in \cite{KBSBook})\[f_\alpha(\lambda;\omega_0\to\omega):=-2\pi\iu\lambda^{-1/2}R_\lambda^{(\alpha)}(\omega,\omega_0)=-\frac{4\pi^2\alpha}{\sqrt{\lambda}}\frac{u_\lambda(\omega)\overline{u_\lambda(\omega_0)}}{G(\alpha,\lambda)}\] for almost all $\lambda\in(0,\infty)$ and $(\omega_0,\omega)\in S^2\times S^2.$\\ The following results on the behaviour of the total scattering cross-section and the scattering amplitude are immediate consequences of Theorem \ref{Rthm}.
\begin{corollary}For any fixed $h\in \mathbb R,$  \begin{enumerate}[label=(\alph*)]
    \item 
$\lim\limits_{\alpha\to\alpha_0}\overline{\sigma}_\alpha(\lambda_h(\alpha))=\dfrac{4\pi}{\lambda_0}\dfrac{1}{h^2+1}$
\item   there exists a sequence $\{\alpha_k\}$ converging to $\alpha_0$ such that \[\lim_{k\to\infty}f_{\alpha_k}(\lambda_h(\alpha_k);\omega_0\to\omega)=-\frac{4\pi}{\sqrt{\lambda_0}}\dfrac{{u}'_{\lambda_0}(\omega)\overline{{u}'_{\lambda_0}(\omega_0)}}{|||u'_{\lambda_0}|||^2}\frac{1}{h+\iu}\quad\text{for a.e. }(\omega_0,\omega)\in S^2\times S^2\]
\end{enumerate} 
\end{corollary}
  Further, if $u\in\mathcal S(\RR^3)$ then $\hat u\in\mathcal S(\RR^3)$ and by \eqref{iso}, the amplitude $f_\alpha(\lambda;\omega_0\to\omega)$ is continuous in $(0,\infty)\times S^2\times S^2.$ We make the following  remark on the convergence of the  amplitude.
 \begin{remark}Suppose $u\in\mathcal S(\RR^3)$ and  $u'_{\lambda_0}\neq 0$ in $L^2(S^2)$. Then, for any fixed $h\in\RR$ \[\lim_{\alpha\to\alpha_0}f_\alpha(\lambda_h(\alpha);\omega_0\to\omega)=-\frac{4\pi}{\sqrt{\lambda_0}}\dfrac{{u}'_{\lambda_0}(\omega)\overline{{u}'_{\lambda_0}(\omega_0)}}{|||u'_{\lambda_0}|||^2}\frac{1}{h+\iu}\quad\text{uniformly for }\omega_0,\omega\in S^2\] 
When  $u_\lambda$ vanishes up to order $n$ at $\lambda_0$, then a similar limiting result holds with $u'_{\lambda_0}$ replaced by ${u}^{(n)}_{\lambda_0}.$
\end{remark}
Next, we study the average time delay for the scattering system $(H_0,H_\alpha)$. For $\alpha\neq\alpha_0$, the average time delay $\zeta_\alpha (\lambda)$  at ``energy" $\lambda$ is shown to be (see \cite{KBS4}) \[\zeta_\alpha(\lambda)=\frac{\diff}{\diff\lambda}\arg\left(\det S^{(\alpha)}_\lambda\right)\]  Hence by \eqref{det S}, we have the following relation:   \begin{equation}\label{TD1}\zeta_\alpha(\lambda)=-2\pi\xi'_\alpha(\lambda)\end{equation}

 In the following theorem, we obtain the behaviour of average time delay $\zeta_\alpha$  near $\lambda_0$ as $\alpha$ goes to $\alpha_0.$  \bthm\label{31} For each fixed $h\in\RR,$ \[ \lim_{\alpha\to\alpha_0}\kappa(\alpha)\zeta_\alpha(\lambda_h(\alpha))=\frac{2}{h^2+1} \]  \ethm\begin{proof} By \eqref{TD1}, note  that  \[\zeta_\alpha(\lambda)=-2\pi\xi'_\alpha(\lambda)=-2\dfrac{G_1(\alpha,\lambda)\frac{\partial G_2}{\partial\lambda}(\alpha,\lambda)-G_2(\alpha,\lambda)\frac{\partial G_1}{\partial\lambda}(\alpha,\lambda)}{|G(\alpha,\lambda)|^2}\] Now proceeding as in proof of Theorem~\ref{ssf thm} and using \eqref{F1F2 1} together with $\dfrac{\partial G_1}{\partial\lambda}(\alpha_0,\lambda_0)=\alpha_0||\phi||^2,~\dfrac{\partial G_2}{\partial\lambda}(\alpha_0,\lambda_0)=0$, the result follows.\end{proof} 
\section{Appendix}
\blem\label{Halmos}  Let $(X,\Sigma,\mu)$ be a measure space and $f_n,f:X\to\RR$ be $\Sigma$-measurable functions satisfying the following conditions: \begin{enumerate}[label=(\alph*)]
\item $f_n,f\geq0$
    \item $\int_Xf_n\diff\mu\leq\int_Xf\diff\mu<\infty$
    \item $f_n\to f$ pointwise almost everywhere.
\end{enumerate} Then \[\lim_{n\to\infty}\int_X|f_n-f|\diff\mu=0\] \elem
\begin{proof}
Setting $g_n = f_n - f$ and  its positive and negative parts $g_n^+$ and $g_n^-$ respectively, $g_n^-=\max\{0,-g_n\}\leq f.$ So by the dominated convergence theorem  \[\lim_{n \to \infty}  \int_X  g_n^{{-}}\diff\mu = 0\]
By the condition (b), we have
\[0\geq\int_X  g_n\diff\mu = \int_X  g_n^+\diff\mu - \int_X  g_n^{-}\diff\mu\] which implies
\[\int_X  g_n^+\diff\mu\leq \int_X  g_n^-\diff\mu\ \text{and}\ \lim_{n \to \infty}  \int_X  g_n^{{+}}\diff\mu = 0\]Thus, \[\lim_{n \to \infty}  \int_X  |g_n|\diff\mu=\lim_{n \to \infty}  \int_X g_n^+\diff\mu+\lim_{n \to \infty}  \int_X  g_n^-\diff\mu=0\]
This completes the proof.
\end{proof}
 \section*{Acknowledgements}
\noindent The first author acknowledges the support received from NBHM (under DAE, Govt. of
India) Ph.D. fellowship grant 0203/7/2019/RD-II/14855.
\bibliographystyle{alpha}

\end{document}